\documentclass[a4paper,11pt,leqno]{amsart}
 
\usepackage{amssymb,hyperref}%,showkeys}

\theoremstyle{plain}

\newtheorem{theo}{Theorem}[section]

\newtheorem{cor}[theo]{Corollary}

\newtheorem{proposition}[theo]{Proposition}

\newtheorem{theorem}[theo]{Theorem}
\newtheorem{lemma}[theo]{Lemma}
\newtheorem{remark}[theo]{Remark}

\newtheorem{fact}[theo]{Fact}
\newtheorem{definition}[theo]{Definition}
\newtheorem{sublemma}[theo]{Sublemma}

\newcommand{\lieo}{\ensuremath{\mathfrak{o}}}
\newcommand{\liek}{\ensuremath{\mathfrak{k}}}
\newcommand{\liep}{\ensuremath{\mathfrak{p}}}

\newcommand{\lieg}{\ensuremath{\mathfrak{g}}}
\newcommand{\lien}{\ensuremath{\mathfrak{n}}}
\newcommand{\liel}{\ensuremath{\mathfrak{l}}}
\newcommand{\liea}{\ensuremath{\mathfrak{a}}}
\newcommand{\liem}{\ensuremath{\mathfrak{m}}}

\newcommand{\lieh}{\ensuremath{\mathfrak{h}}}
\newcommand{\lies}{\ensuremath{\mathfrak{s}}}

\newcommand{\RR}{\ensuremath{{\Bbb R}}}

\newcommand{\re}{\ensuremath{{\mathbb
R}}}

\theoremstyle{definition}

\newtheorem{claim}[theo]{Claim}

\theoremstyle{remark}

 \def\RR{{\mathbb R}}

\begin{document}

\title[Actions of Lie groups]{Actions of semisimple Lie groups preserving a degenerate Riemannian metric}

\author[E. Bekkara]{E. Bekkara$^\star$} 
\thanks{${}^\star$ Partialy supported by the project CMEP 05 MDU 641B of the Tassili program. }

 \address{${}^\star$ ENSET-Oran, Algeria}  
\email{esmaa.bekkara@gmail.com}

\author[C. Frances]{C. Frances$^\dagger$}
\address{${}^\dagger$Laboratoire de
Math\'ematiques, Univ. Paris Sud.}
\email{charles.frances@math.u-psud.fr} 

\author[A. Zeghib]{A. Zeghib$^\ddagger$}
 \address{${}^\ddagger$CNRS, UMPA, \'Ecole
Normale Sup\'erieure de Lyon.}
\email{zeghib@umpa.ens-lyon.fr}

%\thanks{We knowledge partial support of the CMEP agreement}

\date{\today}

\setcounter{tocdepth}{1} 

 \maketitle

\begin{abstract}   
We prove a rigidity of the lightcone in Minkowski space. It is (essentially) the  unique space endowed with a degenerate Riemannian metric, of lightlike type, and supporting  an isometric non-proper action of a semi-simple Lie group.

%In this article we study some Rigidity aspect of the
%lightlike structure from the point of view of isometric action, our
%principals results is that a lightlike manifold with a semi-simple
%isometries group is locally a warped product of the light cone with
%a riemannian manifold.

\end{abstract}

%\tableofcontents

%\baselineskip  0,7 cm

\section{Introduction}

Our subject of study here is {\bf lightlike} metrics on smooth manifolds.
First, a lightlike scalar product on a vector space $E$ is a symmetric  bilinear form  $b$ which is positive but non-definite, and  has exactly a 
1-dimensional kernel. If $E$ has dimension $1+n$, then, in some linear coordinates $(x^0, x^1, \ldots, x^n)$, the associated quadratic form 
$q$ can be written $q = (x^1)^2+ \ldots + (x^n)^2$. 
 Now, a lightlike metric $h$  on a manifold $M$ 
is a smooth tensor which  is a lightlike scalar product on the tangent space of each point.

\subsubsection{Characteristic foliation}    The Kernel of $h$ is a  1-dimensional sub-bundle $N \subset TM$, and thus determines a 1-dimensional foliation ${\mathcal N}$, called the {\bf characteristic} (or null, normal, radical, isotropic...) foliation of $h$.  By definition, any null  curve (i.e. a curve  with  everywhere
 isotropic speed)  of $(M, h)$ trough $x$ 
is contained in the {\it null  leaf}   ${\mathcal N}_x$.  
The (abstract) normal bundle of ${\mathcal N}$, i.e. the quotient $TM/N$ is a Riemannian  vector bundle. Conversely, a lightlike metric consists in giving a 1-dimensional foliation together with a Riemannian metric on its normal bundle.

\subsection{Major  motivations} 
%Our goal in this article is twofold:  first promote this lightlike geometry by illustrating  its interplay with many other fields, and secondly  investigate it from 
%the point of view of isometric actions of Lie groups. 

 Lightlike geometry appears naturally in a lot of geometric situations. We list now some natural examples motivating their study.
 
\subsubsection{Submanifolds of Lorentz manifolds}
 Let $M$ be a submanifold in 
 a Lorentz manifold $(V, g)$. The metric $g$ is non-degenerate with signature $-+ \ldots+$.  However, for  a given $x \in M$, the restriction 
 $h_x$ of $g$ to $T_xM$ has not necessarily the same signature.  Two  easy stable situations are those where $h_x$
 is everywhere of Riemannian type ($M$ is spacelike), or $h_x$ is everywhere of Lorentzian type
 ($M$ is timelike). In both  cases, all the submanifold theory 
 valid in the Riemannian context generalizes: shape operator, Gauss and Codazzi equations...
 
  The delicate situation is when $h_x$ is degenerate for any $x$. Because the ambient metric has Lorentz signature,   
 $h_x$ is then lightlike as defined above. Unfortunately, by opposition to the previous cases, these lightlike submanifolds are generally ``to poor"  to generate a coherent extrinsic local metric differential geometry. Let us give examples of interesting  lightlike submanifolds: 
 
 $\bullet$ {\it Horizons of domains of dependence and black holes.} Unfortunately, they have an essential disadvantage:  their lower smoothness. One can believe that smooth horizons are sufficiently rigid to be classifiable (see for instance \cite{Bonsante, Pio, Hawking}).
  
  $\bullet$  {\it Characteristic hypersurfaces of the wave equation.}  There is a nice interpretation of lightlike hypersurfaces in terms of propagation of waves:  a hypersurface is degenerate iff  it is characteristic for the wave equation (on the  ambient Lorentz space) \cite{Friedlander}.  These hypersurfaces enjoy the  nice property that their null curves are geodesic in the ambient space (this is 
  not true for submanifolds of higher codimension). However, no deeper study of their extrinsic geometry seems to be available in the literature.

 $\bullet$ {\it Lightlike geodesic hypersurfaces.} They are characterized 
 by the fact that  their lightlike metrics are basic (see the example \ref{transversalement.riemannien}). They inherit a connection from the ambient space. See \cite{Dambra-Gromov, DMZ, Zeghib.Commentari, Zeghib.GAFA2}, for their use in Lorentz dynamics. 
 
$\bullet$ {\it Degenerate orbits of Lorentz isometric actions.} Let $G$ be a Lie group acting isometrically on a Lorentz manifold $(V,g)$. Then, any orbit which is lightlike at a point is lightlike everywhere, hence yields an embedded lightlike submanifold in $V$.  The problem of understanding these lightlike orbits, and more generally degenerate invariant submanifolds, is essential when studying such  isometric actions.

$\bullet$ {\it Terminology.} We believe that the choose of the  word ``lightlike'' here is widely justified  from the  relationship between lightlike submanifolds and fields on one hand, and geometric as well as physical optics in general  Relativity  on the other hand (see for instance \cite{R-T}). We also think this terminology is naturally adapted 
to our situation here, but less for the general situation of ``singular 
pseudo-Riemannian'' metrics (compare with \cite{Dug, Kup}).

\subsubsection{From submanifolds to intrinsic lightlike geometry.}

In the last example given above, when restricting the action of the Lie group $G$ to a lightlike orbit, we are led to study the isometric action of $G$ on a lightlike submanifold in a Lorentzian manifold.  In fact, one realizes that the submanifold structure is irrelevant in this problem, and the pertinent framework is that of isometric actions on abstract lightlike manifolds.    

The main difficulty when dealing with this intrinsic formulation is that we loose the rigidity of the ambient action: as we will see below, the isometry group of a lightlike manifold can be infinitely dimensional.

%More generally than lightlike metrics, on can consider 
% ``degenerate pseudo-Riemannian''  ones. More precisely, a degenerate pseudo-Riemannian metric, or alternatively a pseudo-Riemannian semi-metric on a manifold 
%$M$ is just a smooth symmetric tensor $h$ of degree 2 with  {\it constant} signature $(n_-, n_0, n_+)$, so that we have   on each tangent space a quadratic form and an orthogonal basis for it,  with $n_-$ (resp. $n_+$) negative (resp. positive) elements, and $n_0$ isotropic ones.   

%The lightlike case is the less degenerate one: $n_- = 0$ and $n_0 =1$, the most easy to handle, and we will restrict.

\subsection{Two fundamental examples} 

We give now two important examples of lightlike geometries, which are in some sense antagonistic.

\subsubsection{The most flexible example: transversally Riemannian flows}

% \subsubsection{General setup: Singular (degenerate)  pseudo-Riemannian metrics}

% $\bullet$ There are few systematic studies of degenerate pseudo-Riemannian (actually even for the non degenerate ones!).  Let us mention the interesting  reference \cite{Gro2}   where M. Gromov  considers  isometric immersions within this framework.   The degenerate case is interesting by itself but also as a intermediate tool  
% for the non-degenerate (i.e. pseudo-Riemannian ) one.

%\subsection{The example of transversally Riemannian flows.}
\label{transversalement.riemannien}

 The linear situation reduces
to the  case of ${\Bbb R}^{0,n}$, i.e  ${\Bbb R}^{1+n}$  with coordinates $(x^0, x^1, \ldots, x^n)$, 
   endowed with the lightlike  quadratic form $q=(x^1)^2 +
\ldots + (x^n)^2$.  

We will denote its (linear) {\bf orthogonal group}  by  $O(0, n)$ (this is somehow natural since reminiscent of the notation $O(1, n)$). 
We have: 

$$O(0,n)=\left\{ \left(
\begin{array}{cc}
\lambda  &
\begin{array}{ccc}
a_{1} & ... & a_{n}%
\end{array}
\\
\begin{array}{c}
0 \\
. \\
. \\
0%
\end{array}
& A%
\end{array}%
\right) \in GL(1+n,\mathbb{R}),\  A\in O(n)\right\}
$$  

It is naturally isomorphic to the affine similarity group 
$ {\mathbb R} \times Euc_n =  {\mathbb R}.O(n)\ltimes \mathbb{R}^{n}$ (here 
$Euc_{n}=O(n)\ltimes \mathbb{R}^{n}$ is  the group of rigid motions of
the Euclidian space of dimension $n$).

Let us now  see ${\mathbb R}^{1+n}$ as a lightlike manifold. 
The group  of its affine isometric transformations is $O(0, n) \ltimes {\mathbb R}^{1+n}$.

$\bullet$ Contrary to the non-degenerate case, {\it there is  here  a huge 
group (infinitely dimensional) of non-affine isometries}.  Take any
$$\psi: (x^0, x^1, \ldots x^n) \mapsto (\psi _{1}(x^0, x^1, \ldots, x^n), \psi_2(x^{1},..,x^{n}) ),$$
where $\psi _2 \in
Euc_{n}$, and 
 $\psi _{1}:
\mathbb{R}^{n+1} \to \mathbb{R}$  is a smooth function with 
$\frac{\partial \psi_1 }{\partial x^0} \neq 0$ (in order to get a diffeomorphism). 

More generally, let $(L, g)$ be a Riemannian manifold, and  $M = {\mathbb R} \times L$ endowed 
with the  lightlike metric $0 \oplus g$, that is the null foliation is given by the 
${\mathbb R}$-factor, and the metric does not depend on the coordinate along it. Then, we have also here an infinitely dimensional group 
of isometric  transformations given by: 
$ \psi: (t, l) \in {\mathbb R} \times L \mapsto ( \psi_1(t, l), \psi_2(l))$, where $\psi_2$ is an isometry of $L$, e.g. $\psi_2$ the identity map, and $\frac{\partial \psi_1 }{\partial t} \neq 0$ . 

 Conversely, assume that  the lightlike metric $(M, h)$ is  such that 
there exists a non-singular vector field $X$ tangent to the characteristic 
foliation, which flow preserves $h$ (equivalently, the Lie derivative $L_Xh =0$). Then, locally, there is  a metric splitting $M = {\mathbb R} \times L$ as above.  Observe in fact that any vector field collinear to $X$ will preserve $h$ too, in other words any vector field orienting the characteristic foliation ${\mathcal N}$ preserves $h$.  Let us call the lightlike metric {\bf basic}  in this case (they can also be naturally called,  locally product, or stationary). This terminology is  justified by the fact that 
$h$ is the pull-back by the projection map $M \to L$ of the Riemannian metric on the basis $L$. 

$\bullet$ Recall the classical notion from  the geometric theory foliations: a 1-dimensional foliation ${\mathcal N}$ on a manifold $M$ is  transversally Riemannian (one then says ${\mathcal N}$ is a transversally Riemannian flow), if it is the characteristic 
foliation of some  lightlike metric $h$ on $M$, which is preserved by (local) vector fields tangent to ${\mathcal N}$. Therefore this data is strictly equivalent to giving a locally basic lightlike metric on $M$. Of course, the usual classical definition does not involve lightlike metrics. 
 
 There is a well developed theory of 
transversally Riemannian foliations,  with sharp conclusions in the 1-dimensional case  \cite{Car2, Mol}...

The isometry
group of a basic lightlike metric  contains at least all flows tangent to ${\mathcal N}$
which form an  infinitely dimensional group (surely not so
beautiful). However, these metrics are 
somehow tame, since, at least locally, the metric is encoded in an associated Riemannian one.  Moreover, it t was proved 
by D. Kupeli \cite{Kup} (and reproduced in many other places) that some kinf of Levi-Civita connection exists exactly if   the lightlike metric is basic. The connection is never unique, and so enrichment of the structure is  always in order. Actually, the most useful additional   structure is that of a {\bf screen}, mostly developed  in \cite{Dug}, which allows to develop ``calculus'', and get sometimes invariant quantities (see for instance \cite{Gol}). Nevertheless, there is generally no distinguished screen left invariant by the isometry group, so that this notion will not be helpful for us.

\subsubsection{The example  of the Lightcone in Minkowski space}
 \label{lightcone.example} 

We are going now to consider an opposite situation, where the  isometry group is ``big", though remaining finitely dimensional.  
 Let $Min_{1, n}$ be  the Minkowski space
of dimension $1+n$, that is ${\Bbb R}^{1+n}$ endowed with the form
$q= -x_0^2 + x_1^2+ \ldots x_n^2$. The isotropic (positive) cone, or {\it lightcone}, 
$Co^n$ is the set $\{q (x)= 0, \; x_0 >0\}$.  The metric induced by
$q$  on $Co^n$  is lightlike.  The  group $O^+(1, n)$ (subgroup of
$O(1,n)$ preserving the positive cone) acts isometrically on $Co^n$.
This action is in fact transitive so that $Co^n = O^+(1, n) /
Euc_{n-1}$ becomes a lightlike homogenous space, with isotropy group
$Euc_{n-1} = O(n-1) \ltimes {\Bbb R}^{n-1}$, the group of rigid
motions of the Euclidean space of dimension $n-1$.

A key observation is:

\begin{theorem} (Liouville Theorem for lightlike geometry)  \label{isometry.cone} For, $n \geq 3$, any isometry of $Co^n$ belongs to $O^+(1, n)$. In fact, this is true
even locally for $n \geq 4$: any isometry between two connected open
subsets of $Co^n$ is the restriction of an element of
$O^+(1, n)$.\\
$\bullet$ For $n=3$, the group of local isometries is in one-to-one correspondence with the group of local conformal transformations of ${\bf S}^2$.\\
$\bullet$ For $n=2$, there is no rigidity at all, even globally,
since to any diffeomophism of the circle corresponds  an isometry of
$Co^2$.

\end{theorem}

This theorem, which will be proved in  \S  \ref{preliminaries},  shows in particular  that for $n \geq 3$, $Co^n$ is a homogeneous lightlike manifold 
with  isometry group    $O(1,n)$. 
%The proof of this result will be given in  \S  \ref{preliminaries}. 
Remark that for the sake of simplicity, we will often use the notation $O(1, n)$
to mean its identity component $SO_0(1, n)$, and generally any finite index subgroups of $O(1, n)$. Actually, to be precise, we can say that our geometric descriptions of objects are always given {\it up to a finite cover}. 

 It seems likely that  being homogeneous and having a maximal isotropy $O(0, n-1)$ characterizes the flat case, i.e. ${\mathbb R}^{0,n-1}$, and 
having a maximal unimodular isotropy, i.e. $Euc_{n-1}$, characterizes 
the lightcone. In some sense the lightcone is the maximally symmetric 
non-flat lightlike space, analogous to spaces  of constant non-zero curvature in the pseudo-Riemannian case.

\subsection{Statement of results} The present article contains in particular detailed  proofs of the results announced in  \cite{BFZ1}. Before giving the statements, let us recall that two (lightlike) metrics $h$ and $h^{\prime}$ on a manifold $M$ are said to be {\it homothetic} if $h=\lambda h^{\prime}$, for a real $\lambda >0$. A Lie group acts {locally faithfully} on $M$ if  the kernel of the action is a discrete subgroup.

One motivation of the present work was Theorem 1.6 of \cite{DMZ}, that we state here as follows:

\begin{theorem}[\cite{DMZ}]
Let $G$ be a connected group with finite center, locally isomorphic to $O(1,n)$ or $O(2,n)$, $n \geq 3$. 
 If $G$ acts isometrically on a Lorentz manifold, and has a degenerate orbit with non-compact stabilizer, then $G$ is locally isomorphic to $O(1,n)$, and the orbit is homothetic to the lightcone $Co^n$.
\end{theorem}

 Here, we prove an intrinsic version of this result: 

%where it is showed that, under a reasonable
%condition on the action of $O(1,n)$, or any other simple group $G$,
%the unique G-lightlike homogenous manifold is $Co^n$.

\begin{theorem} \label{simple.transitive}
Let $G$ be a non-compact semi-simple Lie group  with  finite center   acting locally faithfully,
isometrically and {\bf non-properly} on a  lightlike
manifold $(M, h)$.  Assume that $G$ has no factor locally isomorphic to $SL(2,\re)$. Then, looking if necessary at a finite cover of $G$:

\begin{itemize} 
\item{$G=H \times H^{\prime}$, where $H$ is locally isomorphic to $O(1,n)$.}
\item{ $G$ has an orbit which is  homothetic, up to a  finite cover,  to a metric product $Co^n \times N$, where $N$ is a Riemannian $H^{\prime}$-homogeneous manifold. The action of $H \times H^{\prime}$ on $Co^n \times N$ is the product action.}
\end{itemize}
\end{theorem}

 Using this theorem and working a little bit more, we can also handle the case where some factors of $G$ are locally isometric to $SL(2,\re)$, when the action is transitive. The following result can be thought  as  a converse to Theorem \ref{isometry.cone}:
 
 \begin{cor}
 \label{simple.homogeneous}
 Let $G$ be a non-compact semi-simple Lie group with finite center, acting locally faithfully, isometrically, transitively and non-properly on a lightlike manifold $(M,h)$, i.e $M$ is a homogeneous lightlike space $G/H$, with a non-compact isotropy group $H$. Then a finite cover of $G$ is isomorphic to $O(1,d) \times H$, where $d \geq 2$ and $H$ is semi-simple.  The manifold $M$ is, up to a finite cover, homothetic to a metric product  $Co^m \times N$, where $N$ is an $H$-homogeneous Riemannian space.  Moreover $m=d$ when $d \geq 3$, and $m=1$ or $2$ when $d=2$.  The action of $G$ on $M$ is the product action
 
 % is locally isomorphic to $O(1,n)$, $n \geq 3$, or $PSL(2,\re)$. In the first case, $M$ is homothetic to the cone $Co^n$. In the second case, $M$ is homothetic to $Co^2$ or $Co^1$, up to finite cover. 
 
 \end{cor}

%\end{remarks}

%\subsubsection{A relative result} There is a relative version of Theorem \ref{simple.transitive}  proved in 
 %\cite{Karin} and  \cite{DFZ}.   There, one assumes that $G$ is $O(1, n)$
 %or $O(2, n)$. More important, $G$  acts on a (ambient)  Lorentz manifold, and 
 %has a degenerate orbit with a no-compact stabilizer.  It is then proved that this orbit is a lightcone. Actually, this result was  among facts which motivated us to search for their intrinsic version, and led to the present work. 
 
%\subsubsection{Extension to semi-simple groups and non-transitive actions} The following is a (maximal) generalization of the previous theorem. Now, this is dynamical result rather than an ``algebraic'' one since it does not deal with necessarily homogeneous spaces

%\begin{theorem} \label{general}

%Let $G$ be a semi-simple Lie group with  finite center  and no compact factor acting
%faithfully, isometrically and {\bf non-properly} on a  lightlike
%manifold $(M, h)$.

%$\bullet$  
%If the action is transitive, then, up to a finite cover,
%$M$ is isometric to a product of  a cone $Co^n$ ($n=1$ is allowed )
%by a homogeneous Riemannian manifold.

%$\bullet$ In the case of a general action, assume that $G$ has no
%factor locally isomorphic to $SL(2, {\Bbb R})$. Then, $G$ has a factor $G^\prime$ locally  isomorphic to   $O(1, n)$ (for some $n$),
%which acts non-properly.  More exactly, 
%$G^\prime$ has an orbit isometric
%to $Co^n$ (up to a finite cover).

%- The product of all factors not locally isomorphic to $O(1, n)$,
%acts properly on $M$.

%\end{theorem}
The non-properness assumption is essential in the previous theorems. If one removes it, ``everything becomes possible". Indeed, consider a Lie group $L$ and a lightlike scalar product on its Lie algebra $\liel$. Translating it on $L$ by left multiplication yields a lightlike metric on $L$, with isometry group containing $L$, acting by left translations.

It is quite surprising that kind of global rigidity theorems can be proved in the framework of lightlike metrics, which are not rigid geometric structures (see \S \ref{transversalement.riemannien}). Here, it is, in some sense,  the algebraic assumption of semi-simplicity which makes the situation rigid. However,  since  any Lie algebra is  a semi-direct product of a semi-simple and a solvable one, it is natural to start  looking  to  actions of semi-simple Lie groups.

When the manifold $M$ is compact, only one simple Lie group can act isometrically, as shows the:

\begin{theorem}\label{compact}
Let $G$ be a non-compact simple Lie group with finite center, acting isometrically on a compact lightlike manifold $(M,h)$. Then $G$ is a finite covering of $PSL(2,\re)$, and all the orbit of $G$ are closed, 1-dimensional, and lightlike.

\end{theorem}

\subsection{The mixed signature case: sub-Lorentz metrics} This notion will naturally  modelize the situation of general submanifolds 
in Lorentz submanifolds. A {\it sub-Lorentz metric} $g$ on $M$
is a symmetric covariant 2 -tensor, which is at each point, a scalar product of  either Lorentz, Euclidean, or lightlike type. The point is that we allow the type to vary over $M$.  So, if $(L, h)$ is a Lorentz manifold, and 
$M$ a submanifold of $L$, then the restriction on $h$ on $M$
is a sub-Lorentz metric (this fact raises the inverse problem, i.e. isometric embedding of sub-Lorentz metrics in Lorentz manifolds). We think it is worthwhile to investigate the geometry of these natural 
and rich structures (see for instance \cite{Miernowski} for a research  of normal forms of these metrics in dimension 2). 

We restrict our investigation here to an adaptation of our lightlike results to this sub-Lorentz situation.

\subsubsection{Lorentz dynamics}  Recall the three fundamental  examples of Lorentz manifolds having an isometry group which acts non-properly. They are just the universal spaces of constant curvature:
\begin{enumerate}
\item The minkowski space: $Min_{1, n-1} =  O(1, n-1) \ltimes {\mathbb R}^n / O(1, n-1)$

\item The de Sitter space $dS_n = O(1, n)/O(1, n-1)$

\item The anti de Sitter space $AdS_n = O(2, n-1) / O(1, n-1)$

\end{enumerate}

In the case of Minkowski space, the isometry group is not semi-simple.

The Lorentz and lightlike dynamics are unified in the following statement:
\begin{theorem} \label{Sub-Lorentz.Theorem} Let $G$ be a semi-simple group with finite center , no compact factor and 
no local factor isomorphic to $SL(2, \re)$,  acting isometrically non-properly on a sub-Lorentz manifold $M$. 
Then, up to a finite cover, $G$ has a factor $G^\prime$ isomorphic to $O(1, n)$ or $O(2, n)$
and having some orbit homothetic to  $dS_n$, $ AdS_n$ 
or $Co^n$.

\end{theorem}

%\subsubsection{Lorentz dynamics}

  %   Degeneracy of Riemannian metrics, Teichmuller space 
    
 %    Observe that a  lightlike metric $g$  generates a semi-distance
%$d_g$, enjoying all properties of standard distances except the
%"separabilty" one:  $d_g(x, y) =0$ does not imply $x=y$.

%\subsubsection{Other connections}
 % $\bullet$ Degeneracy... 

%   $\bullet$ Duality with sub-Riemannian 
   
 %   Let us note  the following 
 %natural situations. 

%\subsection{Related works}

%\subsection{Plan of the article} Let $G$ be a semi-simple acting (differentiably)
%on $M$. Denote by $\lieg$ the Lie algebra of $G$, and $\liea$
% a  Cartan subalgebra in it.  For $X \in \liea$, let $W_X^s$ and $W_X^u$ be  respectively the {\bf stable}  and {\bf unstable} spaces (they are in fact sub-algebras) of $X$. They are respectively the eigenspaces 
%  associated to negative and positive eigenvalues of $ad_X$. The evaluation of these sub-algebras on a point $p \in M$, are tangent sub-spaces of the $G$-orbit at $p$. The question is whether they really correspond to  Lyapunov stable or unstable vectors for the flow generated by $X$? This is always the case if $M$ is compact, and 
%  for ``non-escaping'' points in the general case.  Once, this is the case, these spaces should be isotropic (in all Lorentz or lightlike cases). This key and simple observation was formulated (in this context) for the first time by N. Kowalsky \cite{}. A starting and important part of our proof here consists in improving and  going deeply in this argument...

\section{Preliminaries} \label{preliminaries}

%\subsection{Some lightlike geometry}
%\mnote{Ghani: Add an introducing paragraph}

\subsection {Proof of Theorem \ref{isometry.cone}.}
 
The metric on $Co^n$ is just the metric $0 \oplus e^{2t}g_{{\bf
S}^{n-1}}$ on $\Bbb R \times {\bf S}^{n-1}$. An isometry $f$ of
$Co^n$ is of the form $(t,x)  \mapsto (\lambda(t,x),\phi(x))$. A
simple calculation proves that $f$ is isometric iff:
$$ \phi^*g_{{\bf S}^{n-1}}= e^{2(t- \lambda(t,x))}g_{{\bf S}^{n-1}}  $$

So, any local isometry of $Co^n$ is of the form $(t,x) \mapsto
(t-\mu(x),\phi(x))$, with $\phi$ a local conformal transformation of
the sphere satisfying $\phi^*g_{{\bf S}^{n-1}}= e^{2 \mu}g_{{\bf
S}^{n-1}}$. Thus, the different rigidity phenomena are  just
consequences of classical analogous rigidity results for conformal
transformations on the sphere.
$\Box$

\subsection{$SL(2, \re)$-homogeneous spaces} Understanding  these spaces is worthwhile in our context, since one can  take advantage of restricting the $G$-action to  small simpler groups, e.g. $SL(2, \RR)$ or a finite cover, which always exist   
    in  semi-simple Lie groups.

%\begin{notation}

\subsubsection{Notations}
Let $SL(2,\mathbb{R})$ be the Lie group of $2 \times 2$-matrices with determinant 
  1. It is known that any one parameter subgroup of
$SL(2,\mathbb{R})$ is conjugate to one of the following:
\[
A^+=\left\{ \left(
\begin{array}{cc}
e^{t} & 0 \\
0 & e^{-t}%
\end{array}%
\right) ,\ t\in \mathbb{R}\right\} ,N=\left\{ \left(
\begin{array}{cc}
1 & t \\
0 & 1%
\end{array}%
\right) ,\ t\in \mathbb{R}\right\} \]  
\[ \textrm{ or \ }K^+=\left\{
\left(
\begin{array}{cc}
\sin t & -\cos t \\
\cos t & \sin t%
\end{array}%
\right) ,\ t\in \mathbb{R}\right\}.
\]%

The corresponding derivatives of $A^+$ and $ N$ at the identity are%
\[
X=\left(
\begin{array}{cc}
1 & 0 \\
0 & -1%
\end{array}%
\right) \; \mbox{and} \;\; Y=\left(
\begin{array}{cc}
0 & 1 \\
0 & 0%
\end{array}%
\right). 
\]%
Together with $Z=\left(
\begin{array}{cc}
0 & 0\\
1 & 0%
\end{array}%
\right) 
$, $X$ and $Y$ span the Lie algebra  $\mathfrak{sl}(2,\mathbb{R})$ and satisfy the
bracket
relations:%
\[
\left[ X,Y\right] =2Y,\ \ \left[ X,Z\right] =-2Z\textrm{ and \ }\left[ Y,Z%
\right] =X.
\]
As usual, we denote by $A$ (resp.  $K$), the subgroup generated by $A^+,-A^+$ (resp.  $K^+,-K^+$). 

Let $Aff(\re)$ be the subgroup 
of upper triangular 
 matrices:  
$$Aff(\re)= A.N= \left\{ \left(
\begin{array}{cc}a & b \\0 & a^{-1}%
\end{array}\right) \in SL(2,\re) \right\}$$  
and  $\mathfrak{aff}(\re)$ its Lie algebra.

Non-connected 1-dimensional subgroups of $Aff(\re)$   can be constructed as follows.  Let   $\Gamma_0$ be a cyclic subgroup of $A$  generated by
an element  $\gamma \in A$. The semi-direct product 
$\Gamma_0 \ltimes N $ is a closed subgroup. Conversely, any closed 
1-dimensional non-connected subgroup of $Aff(\re)$ is obtained like this.

Finally, recall $PSL(2,\re)=SL(2,\re)/_{\{\pm Id\}}$.

The ``classical''  classification of the $SL(2,\mathbb{R})$-homogenous
spaces, allows one to recognize   the lightlike ones. 
\begin{proposition}\label{PSL(2,R)surfaces} (Classification of $SL(2, \re)$-homogeneous spaces)
\begin{enumerate}
\item Any $SL(2,\mathbb{R})$-homogenous space  is
isomorphic  to one of the following:
\begin{enumerate}
\item The circle $
S^{1}$ $= SL(2,\mathbb{R})/_{Aff(\mathbb{R})}$, endowed with its natural  projective structure.

\item The hyperbolic plane $= SL(2,\mathbb{R})/_K$, with its Riemannian metric of constant negative curvature.
 
\item The affine punctured plane: $\mathbb{R}^{2} \setminus 
\{0 \} $ $= SL(2,\mathbb{R})/_N$, equipped with an affine flat connection, together with a lightlike metric.

\item A Hopf affine torus $ \mathbb{R}^{2} \setminus 
\{0 \} / _{\{x \sim a x\}}= SL(2,\mathbb{R})/_ {\Gamma_0.N}$, endowed with a flat projective structure.

\item  A  space $SL(2, \re) /_ \Gamma$, where $\Gamma$
is a discrete subgroup of $SL(2, \re)$. It is locally an Anti de Sitter space, i.e. a Lorentz manifold  with negative constant curvature.

\end{enumerate}
\item Up to homothety, the unique lightlike $SL(2,\mathbb{R})$-homogenous spaces  having a non-compact isotropy  are:
\begin{enumerate}
\item The lightcone $Co^1$, i.e the circle $S^1$ endowed with the null metric.

\item The lightcone $Co^2$, namely $\mathbb{R}^{2}\setminus 
\{0 \} $, endowed with the lightlike
metric $d\theta ^{2}$,  where $\mathbb{R}^{2}\setminus\{0 \}$ is parameterized by the  polar coordinates 
($r,\theta)$. 
\end{enumerate}
%Here we do not consider the trivial case of the circle endowed with a null metric.
\end{enumerate}
\end{proposition}

\begin{proof} The proof of the first part is  standard; we just give  details  in the
lightlike case.

Let $\Sigma$ be be an $SL(2,\re)$-homogeneous space of dimension $\geq 2$ i.e $\Sigma \cong SL(2,\re)/_H$, where $H$ is the stabilizer of some $p\in \Sigma$  and conjugated, as showed above,  to one of the following subgroups: $ K, N, \Gamma_0N$ and $\Gamma$. Let $\lieh$ be the Lie algebra of $H$.
Considering    the isotropy representation $$\begin{array}{cccc}\rho_H:&H&\longrightarrow&T_p(\Sigma)= \lieg /\lieh \end{array} $$

 one observes that  when $H=K$ or $\Gamma_0N$, $\rho_H(H)$ is not conjugated to a subgroup of $O(0,1)$.  Now, if $H=\Gamma$,  then $\rho_H(\Gamma)$ is conjugated to a subgroup of $O(1,2)$. This is just because the Killing form on $\mathfrak{sl}(2,\Bbb R)$ has Lorentz signature. If moreover   $\rho_H(\Gamma)$ is conjugated to a subgroup of $O(0,2)$, then $\rho_H(\Gamma)$ has to be finite. Since the Kernel of the adjoint representation of $SL(2,\Bbb R)$ is finite, we get that $\Gamma$ is finite. 
Therefore the unique lightlike $SL(2,\mathbb{R})$-homogeneous space  of dimension $\geq 2$ with non-compact isotropy 
 is $\mathbb{R}^{2}\setminus 
\{0 \} $. 
%(In particular,  If $H$ is a (cyclic) discrete group in $N$ of $A$, then
%the 3-dimensional quotient $SL(2, \re)/H$ has no $Sl(2, \re)$-invariant lightlike metric).

 In order to check that the lightlike metric has the  form 
 $\alpha d\theta^2 $ (for some $\alpha \in \re_+^*$), one argues as follows. At  $p = (1, 0)$, the vector $X$
 is the unique non-trivial eigenspace of $\rho_N$, and  thus the orbit of $p$ by the  flow $\phi _{X}^{t}$   must coincide with the null leaf  ${\mathcal N}_{(1, 0)}$, which is therefore a radial half-line. The other null leaves 
 are also radial, since they are 
  images of   ${\mathcal N}_{(1,0)}$ by the 
$SL(2,\mathbb{R})$-action. 
By homogeneity, the metric must have the form
  $\alpha d\theta ^2 $.
\end{proof}

\begin{remark} Proposition \ref{PSL(2,R)surfaces} is a special case of Theorem \ref{simple.transitive} where
$G = O(1, 2)$.
\end{remark}

For a latter use, let us state the following fact, 
which follows directly from the previous  description of the lightlike surface $\re^2\setminus 
\{0 \}$.

\begin{fact}
\label{2-orbites}
If $Y$ is isotropic at some $p \in \re^{2}\setminus 
\{0 \}$, then $Y$ vanishes at $p$ and $X$ is isotropic at $p$.
\end{fact}

\subsection{Generalities on semi-simple groups; notations} [See 
for instance \cite{Knapp}] 
Let $G$ be a semi-simple group acting isometrically on $(M,h)$.
 This means that we have a smooth homomorphism $\rho : G \to
Diff^{\infty}(M)$, such that for every $g \in G$, $\rho(g)$ acts as
an isometry for $h$ (i.e $\rho(g)^*h=h$). Let $\lieg$ be the Lie
algebra of $G$. For any $X$ in $\lieg$ , we  will generally use the
notation $\phi_X^t$ instead of $\rho(\exp(tX))$. By a slight abuse of
language, we will also denote by $X$ the vector field of $M$
generated by the flow $\phi_X^t$. 

We get, for every $p \in M$, a homomorphism
$\lambda_p : \lieg \to T_pM$, defined by $\lambda_p(X)=X_p$. The flow $\phi_X^t$ stabilizes $p$ iff $X_p=0$, and we denote by
$\lieg_p$ the Lie algebra of the stabilizer of $p$.

We say that  $X\in \lieg$ is {\it lightlike at $p \in M$} (or {\it isotropic})
(resp. {\it spacelike}) if $h_p(X_p,X_p)=0$ (resp.
$h_p(X_p,X_p)>0$).

 We denote by $\mathfrak{s}_p$ the subspace of
all vectors of $\mathfrak{g}$ which are isotropic at $p \in M$.

  Let $O$ be a lightlike $G$-orbit of some $p  \in M$, that is $O\cong G/{G_p}$, where $G_p$ is the stabilizer of $p$. The tangent space $T_pO$ is identified by $\lambda_p$ to the quotient $\lieg/\lieg_p$. In fact the isotropy representation on $T_pO$ is equivalent to the adjoint representation $Ad$ of $G_p$ on $\lieg /{\lieg_p}$. In particular $G_p$ is mapped, up to conjugacy,  to a subgroup of $O(0,n)$. 

Similarly, the Euclidian space $T_pO/{\mathcal{N}_p}$ is identified to $\lieg/{\mathfrak{s}_p}$, where $Ad: G_p \longrightarrow GL(\lieg/{\mathfrak{s}_p})$ preserves  a positive inner product on $\lieg /\mathfrak{s}_p$, so that $G_p$ acts on $\lieg/{\mathfrak{s}_p}$ by orthogonal matrices. In particular, if we consider the tangent representation: $ad:\lieg_p \longrightarrow End(\lieg/{\mathfrak{s}_p})$, the Lie subalgebra $\lieg_p$ acts by skew symmetric matrices on $\lieg/{\mathfrak{s}_p}$.
We will use the same notation for  the elements of the quotients $\lieg/\lieg_p$ and $\lieg/{\mathfrak{s}_p}$ and their  representatives in the Lie algebra $\lieg$.

 We fix once for all a Cartan involution $\Theta$ on the Lie algebra $\lieg$.
This yields a {\it Cartan decomposition} $\lieg=\liek \oplus \liep
$, $\liek$ (resp. $\liep$) being the eigenspace of $\Theta$
associated with the eigenvalue $+1$ (resp. $-1$).

 We choose $\liea$,
a maximal abelian subalgebra of $\liep$, and $\mathfrak{m}$ the
centralizer of $\liea$ in $\mathfrak{k}$. This choice yields a {\it
rootspace decomposition} of $\lieg$, namely there is a finite family
$\Sigma^+=\{ \alpha_1,...,\alpha_s \}$ of nonzero elements of
$\liea^*$, such that $\lieg= \oplus_{\alpha \in
\Sigma^+}\lieg_{-\alpha} \oplus \lieg_0 \oplus \oplus_{\alpha \in
\Sigma^+}\lieg_{\alpha}$. For every $X \in \liea$,
$ad(X)(Y)=\alpha(X)Y$, as soon as $Y \in \lieg_{\alpha}$.
The Lie
subalgebra $\lieg_0$ is in the kernel of $ad(X)$, for every $X \in
\liea$ and splits as a sum: $\lieg_0=\liea \oplus \mathfrak{m}$. 

 The
positive Weyl chamber $\liea^+ \subset \liea$,
contains those $X \in \liea$, such that $\alpha(X) \geq 0$, for all $\alpha \in \Sigma^+$. Its image by the exponential map is denoted by $A^+$.
Let $\Sigma^- = \{ -\alpha_1,...,-\alpha_s   \}$.

The stable subalgebra ( for $\liea$) $W^s=\oplus_{\alpha \in
\Sigma^-}\lieg_\alpha$, and the unstable  one
$W^u=\oplus_{\alpha \in \Sigma^+}\lieg_\alpha$ are both nilpotent
subalgebras of $\lieg$, mapped diffeomorphically by the exponential
map of $\lieg$ onto two subgroups $N^+ \subset G$, and $N^- \subset
G$. 

Given $X \in \liea$, its {\bf stable} algebra is 
$W^s_X=\oplus_{\alpha(X) <0}  \lieg_{\alpha}$, and its {\bf unstable} algebra is
$W^u_X=\oplus_{\alpha (X) >0 } \lieg_{\alpha}$.

% For $X \in \liea$,
%$\Sigma_X^+= \{  \alpha \in \Sigma^+ \cup \Sigma^- \ | \ \alpha(X) >0 %\}$
%(resp. $\Sigma_X^-= \{  \alpha \in \Sigma^+ \cup \Sigma^- \ | \ %\alpha(X) <0 \}$)\\

%The
%subalgebra $\liea$ (resp. the set $\liea^+$) is mapped
%diffeomorphically, by the exponential map of $G$, on an abelian
%subgroup $A \subset G$ (on $A^+ \subset A$).
Let us prove now a lemma which will be useful in the sequel:

\begin{lemma}\label{nilpotantstabilizer}The subalgebra
$W^s_{X}$ has the following properties:

\begin{enumerate}

\item $\left[ \mathfrak{g},W^s_{X} \cap \mathfrak{g}_p \right] \subset \mathfrak{s}_p$

\item $\left[
\mathfrak{s}_p,W^s_{X}\cap \mathfrak{g}_p \right] \subset
\mathfrak{g}_p$

\end{enumerate}
\end{lemma}

\begin{proof} Let $Y \in W^s_{X} \cap \mathfrak{g}_p$
\begin{enumerate}

\item Since $Y \in \mathfrak{g}_p$, $ad_Y$ acts on  $\mathfrak{g}/\mathfrak{s}_p$ by a skew symetric endomorphism, which is moreover nilpotent since $Y \in W^s_X$. Hence $ad_Y$ acts by the null endomorphism on $\mathfrak{g}/\mathfrak{s}_p$, which means that  $ad_Y$ maps $\mathfrak{g}$ to
$\mathfrak{s}_p$.

\item  $ad_Y$ acts as a  nilpotent endomorphism  of $\mathfrak{g}/\mathfrak{g}_p$ (identified with the tangent space),  and has 
$\mathfrak{s}_p/ \mathfrak{g}_p$ (identified to the isotropic direction)
as a 1-dimensional eigenspace. By nilpotency, the action on it is trivial, i.e $ad_Y$ maps  $\mathfrak{s}_p$
into   $\mathfrak{g}_p$. 

\end{enumerate}\end{proof}

Finally, recall that  a semi-simple Lie group of finite center 
admits a Cartan decomposition $G=KAK$, where $K$ is a maximal compact subgroup 
  of $G$.

%The Lie algebra $\lieg$ being semi-simple, can be written as $\lieg=
%\lieg_1 \oplus ... \lieg_r$, where each $\lieg_i$ is a simple Lie
%subalgebra of $\lieg$, which is an ideal in $\lieg$. In the whole
%article, we will do the assumption that {\bf $\lieg$ has no local
%factor isomorphic to ${\mathfrak{sl}(2,\re)}$}, what means that none
%of the $\lieg_i's$ is isomorphic to ${\mathfrak{sl}(2,\re)}$.

\subsection{Non-proper actions} (See  for instance \cite{Kobayashi} for a recent survey about these notions).

\begin{definition} Let $G$ act on $M$. A sequence $(p_k)$
is {\bf non-escaping} if there is a  sequence of transformations $g_k \in G$, such that, both $(p_k)$ and $(q_k) = (g_k(p_k))$ lie in a compact subset of $M$, but $(g_k)$ tends to $\infty$ in $G$, i.e. leaves any compact of $G$.  
 
 -- The sequence $(g_k)$ is  called a {\em ``return sequence''} for $(p_k)$.
 
 -- In the sequel, we will sometimes assume that $(p_k)$ and 
 $(q_k) $ converge to $p$ and $q$ in $M$.
 
One says that the group $G$ acts non-properly  if it admits a non-escaping  sequence.
 \end{definition}

A nice criterion for actions of   semi-simple Lie groups of finite  center to be non-proper, is the next:

\begin{lemma} \label{action de l'algebre de cartan}
Let $G$ be a non-compact semi-simple group with finite center. Then $G$ acts non-properly iff  any Cartan subgroup  $A$ acts non-properly.
\end{lemma}

\begin{proof}
 $G$ admits a Cartan decomposition $KAK$, where $K$ is compact. 
 Let $(p_k)$ be a non-escaping sequence of the $G$-action, and 
 $(g_k)$ its return sequence.  Write $g_k = l_k a_k r_k \in KAK$. Then, $p^\prime_k = r_k(p_k)$ is a non-escaping sequence 
 for the $A$-action, with associated return sequence $(a_k)$. Obviously $(a_k)$ goes to infinity in $A$ since $(g_k)$ goes to infinity in $G$.
  \end{proof}

\section{A Key fact on the stable space}

Here we state   a crucial ingredient  for the proofs of all our theorems. In all what follows, $G$ is a non-compact semi-simple Lie group with finite center
acting  locally faithfully, 
non-properly and  isometrically on a lightlike 
manifold $(M, h)$. The main result of this section is:

\begin{proposition} \label{Key}
If no factor of $G$ is locally isomorphic to $SL(2,\re)$, there exists a Cartan subalgebra $\mathfrak{a}_0$ such that for some $X_0 \in
\mathfrak{a}_0$ and $p_0 \in M$, both  $X_0$ and its stable algebra 
$W^s_{X_0}$ are isotropic at $p_0$.
\end{proposition}

\subsection{Starting fact}

The non-properness of the  action of $G$ leads to a fundamental fact, 
already observed in  \cite{kowalsky}  for Lorentzian metrics, which is the
existence of $p \in M$ and $X \in \liea$
such that    $W^s_{X}$ is isotropic at 
$p$. Let us recall its proof.

\begin{proposition} \label {Kowalsky} \cite{kowalsky}
Let $\liea$ be a Cartan subalgebra of $\lieg$. 
\begin{enumerate}
\item If the flow of $X \in \liea$ acts non-properly, and $p_k \to p$ 
is a non-escaping sequence for the action of $\phi_X^t$, 
 then the stable space $W^s_X$ is isotropic at $p$.

\item More generally, if  $p_k \to p$ is 
a non-escaping sequence for the $A$-action, then  there exists $X \in \mathfrak{a}$ such that 
$W^s_{X}$ is isotropic at  $p \in M$.
\end{enumerate}
\end{proposition}

\begin{proof}
\begin{enumerate}
\item  Denote $\phi_X^t = exptX$ the flow of $X$, and let $(t_k)$ be
a return time sequence  for $(p_k)$, i.e. $\phi_X^{t_k}$ is a return  sequence of $p_k$, what means that  $q_k = \phi_X^{t_k}(p_k)$ 
stay in a compact subset of $M$. 

Let  $Y \in \lieg_\alpha$, then $[X, Y] = \alpha(X) Y$, hence for any $x \in M$, $D_x\phi_X^t Y_x =
e^{t \alpha (X)} Y_{\phi_X^t(x)}$. Assume that $\alpha(X)<0$, then:
$$ h_{p_k}(Y_{p_k},Y_{p_k})=h_{q_k}(D_{p_k}\phi_X^{t_k} (Y_{p_k}),D_{p_k}\phi_X^{t_k}(Y_{p_k}))
 =e^{2t_k\alpha (X)}h_{q_k}(Y_{q_k},Y_{q_k})$$
 On the left hand side, passing to the limit yields   $h_{p}(Y_p,Y_p)$.

On the right hand side, since $(q_k)$ lie in a compact set, 
$h_{q_k}(Y_{q_k}, Y_{q_k})$ is bounded. Therefore, since
$\alpha (X) <0$, this right hand term tends to 0, yielding $h_{p}(Y_p,Y_p)=0$. This proves that $W^s_X$ is isotropic at $p$.

%$$h_{\exp t_kX .p_n}(Ad_{\exp t_kX}(Y_p),Ad_{\exp %t_kX}(Y_p))=2e^{t_k\alpha(X)}h_{\exp t_kX .p_k}(Y_p,Y_p)$$

%(recall that $h_p$ is the induced lightlike product on the tangent %space at $p$)

%Since the
%isotropic vectors are collinear, we have that $W^s_{X}$ is isotropic
%at $p$.

\item

% Lemma (\ref{action de l'algebre
%de cartan}), implies the existence of a non-proper sequence
%$(x_k,p_k)$, such that $(x_k)\subset A.$ 

Let $(X_k)$ be a sequence
in $\mathfrak{a} $ such that $exp(X_k)$ is a return  sequence for $(p_k)$. 
Let $\Vert  . \Vert$ be a
Euclidian norm on $\liea$  and, considering if necessary a subsequence, assume  
$(\frac{X_k}{\Vert X_k\Vert})$ converges  to some $X\in \liea$. As
above, one proves that   that $W^s_{X}$ is isotropic at $p$.
\end{enumerate}
\end{proof}

\begin{remark}\label{semisimplesabilizer} This result is nothing but  a generalization of the linear (punctual) easy  fact:  If  a matrix $A$    preserves a lightlike scalar  product, then its corresponding stable and unstable spaces are isotropic. In our particular case, if $X \in \liea \cap \lieg_p$ (i.e. $X$ stabilizes $p$), then both of $W^s_{X}$ and $W^u_{X}$ are isotropic at $p$.
\end{remark}

\subsection{Proof of Proposition \ref{Key}}
The proof   follows from several observations.   The simplest one
 is that  for  lightlike metrics (in contrast with the Lorentz case),  the isotropic direction is unique on each
tangent space $T_pM$.  
Furthermore, it  coincides  with the non-trivial  eigenspace (if any)  of  any
 infinitesimal  isometry fixing $p$. The hypothesis made in Proposition \ref{Key}, that $G$ has no factor locally isomorphic to $SL(2,\re)$ will be only used in Lemma \ref{centerfact}.

\begin{lemma} For any $p \in M$, the subspace of isotropic vectors  $\mathfrak{s}_p$ is a Lie subalgebra of $\mathfrak{g}$. \end{lemma}

\begin{proof}
Let $X,Y \in \mathfrak{s}_p$, and let $\phi_X^t$ be the isometric
flow generated by $X$ on $M$, then:
\begin{displaymath}[X,Y]_p=\lim_{t \rightarrow
0}\frac{1}{t}[d\phi_X^{-t}(Y_{\phi_X^{t}(p)})-Y_p]\end{displaymath}
since $X,Y$ are isotropic at $p$,  their integral curves at $p$
are supported by the  null leaf  ${\mathcal N}_p$,   and thus $Y_{\phi_X^{t}(p)}$ is isotropic. Since $\phi_X^{-t}$ is an
isometry, $d\phi_X^{-t}(Y_{{\phi_X^{t}}(p)})$ is also isotropic.
\end{proof}

\begin{lemma}\label{dimension.orbite.lightlike}
$G$  stabilizes no  $p \in M$.
\end{lemma}
\begin{proof}

Suppose by contradiction  that $G$
stabilizes  $p \in M$, then G acts on $T_{p}M$ by: 
$\rho:g \mapsto d_{p }g \in GL(T_{p}M).$ Since $G$ preserves the lightlike scalar  product $h_p$, it is mapped by
$\rho$ into a
subgroup of $O(0,n)$. Thus, at the level of Lie algebras, we get a homomorphism $d\rho : \lieg \to {\mathfrak o}(0,n)$. Now, we prove:
\begin{sublemma}
Any homomorphism from $\lieg$ to ${\mathfrak o}(0,n)$ is trivial.
\end{sublemma}
\begin{proof}
Without loss of generality,  we can suppose that $\lieg$ is simple. Let
$\lambda$ be a homomorphism from $\lieg$  to ${\lieo}(0,n)$, $\pi$ the
projection from ${\lieo}(0,n)$ to ${\lieo}(n)$. Consider the homomorphism
$\lambda \circ \pi: \lieg \longrightarrow {\lieo}(n)$. Since $\lieg$ is simple and
noncompact, it has no non-trivial homomorphism into the Lie algebra of a compact group;
this implies that $\lambda \circ \pi$ is trivial. So, $\lieg$ is mapped by
$\lambda$ into the kernel ${\lieg}_{0}$ of $\pi $, that is the algebra of matrices of
the form $\left(
\begin{array}{cc}
\mu  &
\begin{array}{ccc}
x_{1} & ... & x_{n}%
\end{array}
\\
\begin{array}{c}
0 \\
. \\
. \\
0%
\end{array}
& 0%
\end{array}%
\right). $ Since  ${\lieg}_{0}$ is  solvable and
$\lieg$ is simple, we conclude that $\lambda $ is trivial.
\end{proof}

As a corollary, the $\rho$-image of any connected compact subgroup $K \subset G$ is trivial. However such $K$ preserves a Riemannian metric. 
But on a connected manifold $M$, a 
Riemannian isometry which fixes a point and has a trivial
derivative at  this point, is the identity on $M$. This is easily seen since in the neighbourhood of any fixed point, a Riemannian isometry is linearized thanks to the exponential map.
Hence $K$ acts trivially on $M$, and therefore $G$ does not act faithfully, which contradicts  our hypothesis and completes the proof of our lemma.
\end{proof}

\begin{lemma}\label{centerfact}
If $G$ has no factor locally isomorphic to $SL(2,\re)$, then no Cartan subalgebra $\liea$  meets  the stabilizer  
subalgebra: $\liea \cap \lieg_p = \{ 0
\}$ for any $p \in M$.
\end{lemma}

\begin{proof}
Assume by contradiction that,
$\liea \cap \lieg_p \not = \{  0  \}$, and let us take $X \not = 0$
in this intersection.  Apply  Remark \ref{semisimplesabilizer} to
$X$ to get that the subspaces $W^s_{X}$ and $W^u_{X}$ are both isotropic at
$p$.  It is a general fact that $\lien$, the Lie subalgebra generated by  $W^s_{X}$ and $W^u_{X}$, is an ideal of $\lieg$  (see for instance \cite{kowalsky}), and it is in particular 
a factor of $\lieg$. It acts on the 
   1-manifold    ${\mathcal N}_p$. This action is faithful, otherwise its kernel $\lies$ would be the Lie algebra of a semi-simple group $S \subset G$, which would have fixed points on $M$, in contradiction with Lemma \ref{dimension.orbite.lightlike}. Now, 
 the 
only the semi-simple algebra acting  faithfully
on a 1-manifold is $\mathfrak{sl}(2,\re)$. This contradicts our
hypothesis that $\lieg$ has no such factor.
 \end{proof}

\begin{lemma}
Let $H$ be a Lie group having $\mathfrak{sl}(2,\re)$ as a Lie algebra.
\begin{enumerate}
\item If $H$ is linear, then it is isomorphic to either $SL(2,\re)$ or $PSL(2,\re)$.
\item If $H$ is a subgroup of a Lie group $G$ with finite center, then it is a finite covering of $PSL(2,\re)$.

\end{enumerate}
\end{lemma}
\begin{proof}
The point is that all the representations 
of the Lie algebra $\mathfrak{sl}(2, \re)$ integrate to  actions 
of the group 
$SL(2, \re)$  itself (and not merely its universal cover).  Indeed, the classical classification asserts that all the irreducible representations are isomorphic to symmetric powers of the standard representation, or equivalently to representations on spaces of  homogeneous polynomials 
of a given degree,  in two variables $x$ and $y$ (see for instance \cite{Knapp}).   Clearly, $SL(2, \re)$
acts on these polynomials,  and  $PSL(2, \re)$ acts iff the degree 
is even. 
 For the last point, observe that the adjoint representation of $G$ 
has finite Kernel. 
\end{proof}

\subsubsection*{End of the proof of Proposition \ref{Key}}

From Proposition \ref{Kowalsky}, there exist $X \in \liea$
and    $p \in M$, such that $W^s_X$ is isotropic at $p$. Since
$\mathfrak{g}$ has no local factor isomorphic to
$\mathfrak{sl}(2,\re)$,  we have $\dim  W^s_X>1$ (otherwise 
the subalgebra $ \liea \oplus \Sigma_{\alpha(X) \geq 0} \lieg_\alpha$ 
is supplementary to $W^s_X$ and would have codimension 1, and therefore $\lieg$ acts on a 1-manifold).
 For a lightlike metric,  an isotropic space 
has dimension at most 1, so that the evaluation of $W^s_X$ at $p$
has at most dimension 1, and thus $W^s_X$ contains at least a non-zero vector 
$Y_0$ vanishing at $p$. 
  
By the Jacobson-Morozov Theorem (see \cite{Knapp}), the nilpotent element $Y_0$ belongs to some
subalgebra $\lieh$ isomorphic to ${\mathfrak sl}(2, \re)$, i.e. generated by an $sl_2$-triple $\{X_0, Y_0, Z_0\}$,  such that $[X_0,Y_0]=2Y_0$,$[X_0,Z_0]=-2Z_0$, and $[Y_0,Z_0]=X_0$.

 Let $H \subset G$ be  the group  associated
  to $\mathfrak{h}$.  From the lemma above and the fact that $G$ has finite center, $H$ is a finite covering of $PSL(2,\re)$.
Let us call $\Sigma$ the $H$-orbit of $p$. Since $Y_0$ vanishes at $p$, but not 
$X_0$ (by Lemma \ref{centerfact}), $\Sigma$ is a lightlike surface,  homothetic up to finite cover to $(\re^2 \setminus \{ 0\},d\theta^2)$. Also, $ Y_0 \in \mathfrak{g}_p$ implies that $H$ acts non-properly on $\Sigma$. The group  $\exp(\re X_0)$ is a Cartan subgroup of $H$, and by Lemma \ref{action de l'algebre de cartan}, $\exp(\re X_0)$ also acts non-properly. Thus, we can find $(q_k)$ a sequence of $\Sigma$ converging to   $p_0 \in \Sigma$, and a sequence of return times $(t_k)$, such that $h_k.q_k$ converges in $\Sigma$, where $h_k=\exp(t_kX_0)$.  Because in any finite-dimensional representation of ${\mathfrak sl}(2, \re)$, any $\re$-split element is mapped on some $\re$-split element, the Cartan subalgebra $\re X_0$ is contained in a Cartan subalgebra of the
 ambient algebra $\lieg$, say
 $\mathfrak{a}_0$. Now, we apply the first part of Proposition \ref{Kowalsky} to $X_0$ and $\liea_0$, and deduce 
  that  $W^s_{X_0}$ is isotropic at $p_0$ (where $W^s_{X_0}$ is  defined relatively  to $\mathfrak{a}_0$). In particular, $Y_0$ is isotropic at $p_0$, and Fact \ref{2-orbites} then ensures that $X_0$ is also isotropic at $p_0$.
$\Box$

\section{Proof of Theorem \ref{simple.transitive}}

%Keeping the notations  of the last section, we have the following starting fact. 

\subsection{Reduction lemma} The following  fact reduces the proof of Theorem \ref{simple.transitive} to the case of non-proper {\it transitive}  actions of semi-simple  groups.

%As a corollary, we have the following.

%  If the action of G is not transitive

%Suppose here that $G$ acts non-transitively on $M$, the following
%proposition shows the existence of {\it non-proper} lightlike orbits on %$M.$ This
%reduces the study of general non-transitive action of $G$, to the
%study of the transitive one.

\begin{lemma} {\em (Reduction to  the transitive case)}  \label{reduction.transitive}
Let $G$ be a semi-simple Lie group with no factor locally isomorphic to $SL(2, \re)$, acting faithfully non-properly isometrically on a lightlike manifold. Then 
there exists a $G$-orbit which is    lightlike  and the $G$-action on 
it is non-proper, i.e. has  non-precompact stabilizers. In fact, the stabilizer subalgebras contain nilpotent elements.
 
 More precisely,    any $p \in M$, for which there exists  $X$ such that  $W^s_X$ is isotropic at $p$, has a lightlike orbit on which the action is non-proper.

%further more For any $p\in M$, such that a stable space $W^s_X$ is isotropic
%at $p$, the $G-$orbit $O$ of $p$, is lightlike.
\end{lemma}

\begin{proof}  We have already seen  at the   end of the proof of Proposition \ref{Key}, that $W^s_X$  has dimension $>1$.
If it is isotropic at $p$, then it contains  elements  $Y \in W^s_X \cap \lieg_p$ vanishing at $p$. But $Y$ is a nilpotent element in $\lieg$, and in particular $Ad(\exp(tY))$ is non-compact, proving that the stabilizer of $p$ is non-compact.  Therefore the action of $G$ on the $G$-orbit $O$ of $p$   is non-proper.   

Let us show that $O$ is lightlike. Lemma \ref{dimension.orbite.lightlike} shows  that\ $O$ can not be reduced to $p$. Also, $O$ can not be 1-dimensional. Indeed, a factor of $G$ should then act faithfully on $O$, and we already saw that such a factor would be locally isomorphic to $SL(2,\re)$.

Suppose now by contradiction that $O$ is 
Riemannian. Then any vector which is isotropic at $p$ must vanish there, in particular $W^s_X \subset  \mathfrak{g}_{p}$.

Consider  $Y \in W^s_X$,  and its   infinitesimal action on the tangent space 
of the orbit at $p$. This action is just  $
ad(Y): \mathfrak{g}\mathbf{/}\mathfrak{g}_{p} 
\longrightarrow 
\mathfrak{g}\mathbf{/}\mathfrak{g}_{p}$. 
If $O$ is supposed to be Riemannian, it is at the same time skew symmetric and nilpotent,  hence trivial  
   (on $\mathfrak{g}
\mathbf{/}\mathfrak{g}_{p}$), which means:  $ad(Y)\left(
\mathfrak{g}\right) \subset \mathfrak{g}_{p}$, for all $Y \in W^s_X$. Now, let us pick $Y_0 \in W^s_X$, and use the Jacobson-Morozov theorem to get an ${\mathfrak sl}(2,\re)$-triple $(Z_0,X_0,Y_0)$. Then $ad(Y_0)(Z_0)=X_0$, so that $X_0 \in \lieg_p$. But we saw that $X_0$ is in a Cartan subalgebra of $\lieg$, yielding a contradiction with  Lemma \ref{centerfact}.

%It turns out that   $\mathfrak{a}\subset 
%ad_{W^s_X}\left( \mathfrak{g}\right) $,  and thus, 
%$\mathfrak{a}\subset\mathfrak{g}_{p} $, which contradicts Lemma   
 %\ref{centerfact}. For our purpose, in order to get 
  %a contradiction with Lemma \ref{centerfact},  we just need to check that $\mathfrak{a} \cap 
%ad_{W^s_X} (\mathfrak{g}) \neq 0$.  This fact follows, for example, from 
%the fact that $[Y, \Theta (Y)]$ is a non-trivial element of $\liea$ (where 
%$\Theta$ is the Cartan involution).
  \end{proof}

\subsection{Proof in the simple case}
 We now give  the proof of Theorem \ref{simple.transitive}, assuming that the group $G$ is  simple with finite center, and the action is transitive and non-proper. The general case of semi-simple groups will be handled in the next section. The proof will be achieved in several steps. 
%If $G$ is a finite covering of $PSL(2,\re)$, it follows immediately from Proposition \ref{SL(2,R)surfaces} that $(M,h)$ is either a circle with the null metric, or the lightcone $Co^2$ (up to homothety). 

%Let us assume now that $G$ is not locally isomorphic to $SL(2,\re)$. The proof will be achieved in several steps.

\medskip
\textbf{Step 1}: There exist $p \in M$ and $X$ in some Cartan
subalgebra $\liea$, such that  $W^s_X \subset \lieg_p$.

\begin{proof}
Proposition \ref{Kowalsky} says  that for some  $p\in M
$, both  $X$ and $W^s_X$ are isotropic at $p$.
For any
 $Y $ in $ W^s_{X}$,  the Lie algebra generated by $X$ and $Y$ is isomorphic to the Lie algebra $\mathfrak{aff}(\re)$ and acts on the null leaf  ${\mathcal N}_p$. 
 %To simplify notation,  let us assume that the group $Aff(\re)$ itself  acts on the  1-manifold    ${\mathcal N}_p$, with a connected stabilizer, that  is,  a one parameter group of $Aff(\re)$.
 %It is known that any one parameter group 
 %of $Aff(\re)$ is conjugate to that generated by $X$ of $Y$. 
 Up to isomorphism,  there are exactly two actions 
 of $aff(\re)$ on a connected 1-manifold:
 
 \begin{enumerate}

\item The usual affine action of $aff(\re)$ on the line. Here, a conjugate of $X$ vanishes somewhere.  %which contradicts  Lemma \ref{centerfact}.

\item The non-faithful action, for which $Y$ acts trivially.  

\end{enumerate}

We can not be in the first case without contradicting  Lemma \ref{centerfact}, so that only possibility (2) occurs, and thus  $W^s_X \subset \lieg_p$.
\end{proof}
\textbf{Step 2}: The $\re$-rank of $\mathfrak{g}$  equals  $1$.

\begin{proof}
Suppose  the $\re$-rank of $\mathfrak{g}$  $ >1$.   Let $\alpha$ be a root such that $\alpha (X) > 0$ and  $\beta$  an adjacent root in the Dynkin diagram, according to the choice of a basis $\Phi$ of positive simple roots where,  $\gamma \in \Phi \Longrightarrow \gamma(X) \ge 0$.  (See \cite{Knapp}).

By definition,  $\alpha + \beta$ is also a root and  $(\alpha +\beta) (X) > 0$, that is,  $\lieg_{-\alpha}$ and $\lieg_{- (\alpha+\beta)}$ are different 
and contained in $W^s_X$.

%\mnote{ V\'erifier et donner reference pour ce Step. }

 Let $T_{\alpha}$ and $T_{\alpha + \beta}$ be the vectors 
 of $\liea$ dual to $\alpha$ and $\alpha +\beta $ respectively. They are linearly independent. 
 
  Moreover,  $T_\alpha \in [\lieg_{\alpha},  \lieg_{-\alpha}] \subset ad_{\mathfrak{g}} (W^s_{X})$, and similarly 
 for $T_{\alpha + \beta}$,
 
 % and therefore, 
%  $T_{\alpha}, T_{\alpha + \beta} \in ad_{\mathfrak{g}} (W^s_{X})$. 

 By the first step and Lemma \ref{nilpotantstabilizer}, $T_\alpha$ and $T_{\alpha+ \beta}$
 are isotropic at $p$. Hence, there is a non-trivial linear combination of them which vanishes at $p$. 
   This  contradicts Proposition \ref{centerfact} claiming that 
   $\liea \cap \lieg_p = 0$. 
   Therefore, $\lieg$  has rank  1.
\end{proof}

\begin{remark}  It is exactly here that we need $G$ to be simple!
%\mnote{Confirmer cette remarque}

\end{remark}

\textbf{Step 3}: The Lie algebra $\mathfrak{g}$ is isomorphic to
$\mathfrak{o}(1,n)$

\begin{proof}
 Suppose that $\mathfrak{g}$ is not isomorphic to $\mathfrak{o}(1,n)$,
 then we have two roots $\alpha, 2\alpha$, such that $\alpha(X)>0$, 
 
 \begin{claim}
  The bracket $[ \lieg_{+2\alpha}, \lieg_{-\alpha}] \neq 0$.
\end{claim}
 Let us continue the proof assuming the claim.  Consider a non-zero   $Y \in [ \lieg_{+2\alpha}, \lieg_{-\alpha}] \subset \mathfrak{g}_{+ \alpha} $. By Lemma  \ref{nilpotantstabilizer}, $Y$ is isotropic at $p$.
   Let $\Theta$ be the Cartan involution, then $\Theta Y \in W^s_X$, and hence belongs to $\lieg_p$, by Step 1.  Lemma \ref{nilpotantstabilizer} then implies that
    $[Y,\Theta Y] \in \mathfrak{g}_p$, in particular $\liea \cap \lieg_p \neq 0$, which  contradicts Lemma \ref{centerfact}. 
  \end{proof}

{\it Proof of the claim.} 
First, the rank 1  simple Lie groups  of non-compact type are known 
to be the isometry groups of symmetric spaces of negative curvature. They are the real, complex and quaternionic hyperbolic spaces, together with the hyperbolic Cayley plane.  A direct computation can be performed and yields the claim. Let us give another synthetic proof. By contradiction, the sum $ \liel= \lieg_0 + \lieg_{-\alpha} + \lieg_{-2\alpha}
+ \lieg_{+2 \alpha}$  would be  a subalgebra of $\lieg$. For the sake of simplicity, let us work with groups instead of algebras. Let $L$ 
be the group associated to our last subalgebra$\liel$.  Clearly, $L$ is 
not compact.  The point is that 
there is a dichotomy for non-compact  connected  isometry subgroups 
of  negatively curved symmetric spaces. If they 
have a  non-trivial  solvable radical, then they 
fix a point at infinity, and thus  are contained in  a parabolic group and in particular
have a compact simple (Levi)-part   (see \cite{Ebe}).  If not the group is semi-simple.
It is clear that our $L$ contains  a non-compact semi-simple, and therefore by the dichotomy, it  is semi-simple. But, once semi-simple, $L$
will have a ``symmetric '' root decomposition, i.e. the negative of a
root is a root, too.  Thus, there must exist 
a non-trivial root space corresponding to $\alpha$, which contradicts 
the true   definition of $\liel$.
$\Box$
%\end{proof}

\textbf{Step 4}: The full isotropic subalgebra is 
$\mathfrak{s}_p=\mathfrak{a} \oplus \mathfrak{m} \oplus
\mathfrak{g}_{-\alpha}$

\begin{proof} Recall that $\mathfrak{m}$ is the Lie algebra of the centralizer of $\liea$ in the maximal compact $K$.
 %Since $\mathfrak{g} =
%\mathfrak{o}(1,n)$, it it has a unique positive root $\alpha$,
%let $\mathfrak{g}_{\alpha}$ be its corresponding  eigenspace. 
Since $\mathfrak{m} \subset
[\mathfrak{g}_{\alpha}, \mathfrak{g}_{-\alpha}]$,   Lemma
\ref{nilpotantstabilizer} implies  that it is isotropic at $p$. 

On the other hand, if $Y \in \lieg_+\alpha$ is isotropic at $p$, Lemma \ref{nilpotantstabilizer} implies that  the semi-simple element $[Y,\Theta Y] \in [\lieg_{-\alpha},\lieg_\alpha]\subset [\mathfrak{s}_p,W^s_p \cap \lieg_p]$ is in the the stabilizer subalgebra  of $p$, which contradicts Lemma \ref{centerfact}.  Therefore, 
the isotropic subalgebra is exactly
$\mathfrak{s}_p=\mathfrak{a} \oplus \mathfrak{m} \oplus \mathfrak{g}_{\alpha}$.
\end{proof}

\textbf{Step 5}: The full stabilizer  subalgebra is 
$\lieg_p=\liem \oplus \mathfrak{g}_{\alpha}$.

\begin{proof}
Let $Z \in \mathfrak{m}$, then it is isotropic at $p$. Suppose by contradiction that   $Z \notin \lieg_p$. Then,  there exists an
element $Z + \lambda X \in \lieg_p$,  
$ \lambda \in \re^*$.  We let it act on the normal space
of the  null leaf.

The action of $X$ on $\mathfrak{g}/ \mathfrak{s}_p$  is identified 
to its action on $\lieg_{+\alpha}$, by  the previous step. 
In particular, the $X$-action  has non-zero  real eigenvalues.

The action of $\liem$ (and in particular $Z$) on $\mathfrak{g}/ \mathfrak{s}_p$ has purely imaginary eigenvalues, since $\liem$
is contained the Lie algebra of a maximal compact group. 

On one hand, since $X$  and $Z $ commute (by definition of $\liem$), the action of $Z+\lambda X$ on 
$\mathfrak{g}/ \mathfrak{s}_p$ must have eigenvalues with non-trivial 
real part. 

On the other hand,  $Z + \lambda X \in \lieg_p$  and acts as a skew
symmetric endomorphism  on $\mathfrak{g}/ \mathfrak{s}_p$, and thus has only  purely imaginary eigenvalues: contradiction. This shows that 
$\liem \subset \lieg_p$, but since $\liea \cap \lieg_p = 0$, and 
$\lieg_p \subset \lies_p$, which was calculated in the previous space, we get the equality $\lieg_p = \liem \oplus \lieg_-\alpha$.
\end{proof}

{\bf End:} Since $\lieg$ is isomorphic to $\mathfrak{o}(1,n)$, and the Lie algebra of the stabilizer $\lieg_p$ is isomorphic to the Lie algebra of the group of Euclidian motions $Euc_{n}$, we conclude that $M$ is a covering of the Lightcone in Minkowski space, which completes  the proof of Theorem \ref{simple.transitive} when $G$ is simple.

\subsection{End of the proof}
\label{the-end}
Thanks to Lemma \ref{reduction.transitive}, the complete proof of Theorem{simple.transitive} reduces to the  study of non-proper transitive actions of semi-simple groups with no factor locally isomorphic to $SL(2,\re)$. The work made above will be useful thanks to the following reduction lemma:
 
\begin{lemma} {\em (Reduction to the simple case)} \label{reduction.simple}
 Let $X$ be in a Cartan subalgebra of $\lieg$, such that $W^s_X$ is isotropic at $p$.  Consider the decomposition of $\lieg$ in simple factors. Let $\lieh$ be such a simple factor,  and $H \subset G$ the corresponding  group. 
Suppose $X$ has a non-trivial projection  on $\lieh$. Then the $H$-orbit is non-proper and lightlike.
\end{lemma}

\begin{proof} Write $\lieg=\lieh_1 \oplus... \oplus \lieh_s$, where the $\lieh_i$'s are the simple factors of $\lieg$, and call $X_i$ the projection of $X$ on $\lieh_i$. If $W^s_{X_i,\lieh_i}$ denotes the stable space of  $X_i$ relatively to $\lieh_i$, it is straightforward to check that $X^s_{X}=W^s_{X_1,\lieh_1} \oplus ... \oplus W^s_{X_s,\lieh_s} $.  In particular, if $W^s_X$ is isotropic at $p$ and $\lieh$ is a simple factor on which $X$ has a non-trivial projection $X^{\prime}$, then $W^s_{X^{\prime},\lieh}$ is non-trivial and isotropic at $p$.   We infer from Corollary \ref{reduction.transitive} that
the $H$-orbit of $p$ is lightlike and  the action of $H$ on it is non-proper. 
\end{proof}
%\section{Proofs of Theorems  \ref{general} and \ref{Sub-Lorentz.Theorem}}

%\subsection{Proof of Theorem \ref{general}} Let $G$ be a semi-simple
%group acting as in the hypotheses of the theorem, and assume to begin with that $G$ has no local factor isomorphic to 
%$SL(2, \re)$.

By this lemma, there is a simple factor $H$ of $G$ having a lightlike 
non-proper orbit $H.p$.  It follows from the previous section that $H$ is locally isomorphic to  $O(1, n)$, $n \geq 3$, and  $H.p$ is homothetic to  $Co^n$. There is a semi-simple group $H^\prime$ such that $G$ is a finite quotient of $H \times H^{\prime}$. This product still acts locally faithfully on $M$, so that we will assume $G= H \times H^{\prime}$ in the following. 
Consider $O = G.p$  the $G$-orbit containing this $H.p$. The remaining   part of Theorem \ref{simple.transitive} is the geometric description of $O$: it is  
a direct metric product $Hp \times H^\prime p$ (up to a finite cover). This is the containt of the following lemma, which will be also useful when dealing with groups having factors locally isomorphic to $SL(2,\Bbb R)$.

\begin{proposition}
\label{produit}
Let $G$ be a semi-simple Lie group acting locally faithfully transitively and non properly on a lightlike manifold $(M,h)$. We assume that $G=H \times H^{\prime}$, where $H$ is isomorphic to $O(1,d)$, $d \geq 2$, and $H^{\prime}$ is semi-simple. We assume that for some $p \in M$, the orbit $H.p$ is homothetic to $Co^d$ (resp. $Co^2$ or $Co^1$) if $d \geq 3$ (resp. $d=2$). Then $M$ is a metric product $M=Co^m \times N$, where $N$ is an $H^{\prime}$-homogeneous Riemannian manifold.

\end{proposition}

\begin{proof}
$M$ is naturally foliated by lightcones   ${\mathcal H}_x = H.x$.
This foliation is $G$-invariant: if $g \in G$, then, $g {\mathcal H}_x=  g H.x= Hg.x = {\mathcal H}_{gx}$, since $H$ is normal in $G$.

%1) We claim that $H^\prime$ acts properly on $O$. If not, by contradiction, Theorem \ref{simple.transitive} says that $G$ has a factor $H^{\prime \prime} \subset H^\prime$ isomorphic 
%to $O(1, n^\prime)$, such that $H^{\prime \prime}.p$ is isometric to 
%a cone $Co^{n^{\prime}}$. In this case, the cones $H.p$ and $H^{\prime \prime}.p$
%share at least  the null leaf ${\mathcal N}_p$ (since it is the unique isotropic curve trough $p$). Let us pick  $X$ and $X^{\prime \prime}$ in $\lieh$ and $\lieh^{\prime}$ respectively, such that $exp(tX)$ and $exp(tX^{\prime})$ are hyperbolic flows of $O(1,n)$ and $O(1,n^{\prime})$ leaving ${\mathcal N}_p$ invariant. One can then find a non-trivial linear combination $\lambda X + \mu X^{\prime}$ such that $exp(\lambda X + \mu X^{\prime})$ fixes ${\mathcal N}_p$ pointwise.   This contradicts Lemma \ref{centerfact} since 
%$\re X \oplus \re X^{\prime \prime}$ can be embedded in a Cartan 
%subalgebra of $G$.

1) We first prove that $H^\prime.p$ is Riemannian. If it is not the case, 
it contains the null leaf ${\mathcal N}_p$. If $p^{\prime} \not = p$ is a point of ${\mathcal N}_p$, then   there exists
$h^\prime \in H^\prime$ such that $h^{\prime}.p=p^{\prime}$. Also, $h^{\prime}$ maps the null line of $H.p$ passing through $p$ onto the null line of $H.p^{\prime}$ passing through $p^{\prime}$, which means that $h^{\prime}$ preserves ${\mathcal N}_p$ and acts 
non-trivially on it. But since $p^{\prime}$ is a point of ${\mathcal N}_p \subset H.p$, $h^{\prime}$ maps $H.p$ on itself. Thus, $h^{\prime}$ is an isometry of the cone $H.p$, commuting with the action of $H$.

\begin{lemma}
\label{commute}
Let $h^{\prime}$ be an isometry of the cone $Co^d$, $d \geq 2$ (resp. of $Co^1$), commuting with the action of $O(1,d)$ (resp. $O(1,2)$). Then $h^{\prime}$ is the identity map of $Co^d$ (resp. $Co^1$).
\end{lemma}

\begin{proof}
We begin with the case of $Co^1$. An isometry of $Co^1$ is just a diffeomorphism of ${\bf S}^1$. If such a diffeomorphism commutes with the projective action of $O(1,2)$ on ${\bf S}^1$, it must fix, in particular, all the fixed points of parabolic elements in $O(1,2)$. But the set of these fixed points is precisely ${\bf S}^1$, so that we are done.

Now, in higher dimension, we saw that writing $Co^d$ as ${\Bbb R} \times {\bf S}^{d-1}$ with the metric $0 \oplus g_{{\bf S}^{d-1}}$, the isometry $h^{\prime}$ is of the form: $(t,x) \mapsto (t-\mu(x),\phi(x))$. Here $\phi$ is a conformal transformation of ${\bf S}^{d-1}$ satisfying $\phi^*g_{{\bf S}^{d-1}}=e^{2 \mu}g_{{\bf S}^{d-1}}$. Now, if $h^{\prime}$ commutes with the action of $O(1,d)$, it must leave invariant any line of fixed points of parabolic elements in $O(1,d)$. This yields $\phi(x)=x$, and finally $\mu(x)=0$. 

\end{proof}

The lemma implies that $h^{\prime}$ acts identically on $H.p$, a contradiction with the fact that it acts non-trivially on ${\mathcal N}_p$.

%Hence, by Theorem 
%\ref{isometry.cone}, there is $h \in H=O(1,n)$ such that the actions of $h^{\prime}$ and $h$ on $H.p$ coincide. Since this $h$ preserves ${\mathcal N}_p$ and acts non-trivially on it, this is a loxodromic element of $O(1,n)$. This is not difficult to chek that such elements do not act properly on $Co^n$ (for instance, in dimension 2, 
%in the $\re^2 {\setminus} \{0\}$-model of $Co^2$, $f$ is equivalent to a matrix 
%$ \left(
%\begin{array}{cc}
%\lambda & 0 \\
%0 &  \lambda^{-1}%
%\end{array}%
%\right)$). As a consequence, $H^{\prime}$ does not act properly on $H.p$, contradicting point 1).

%there is also $h$$$ This $h^\prime$ preserves the lightcone leaf 
%${\m\primeathcal H}_p$. Its action on it coincides with that 
%of an element $f$ of the Moebuis group $H = O(1, n)$ (by Theorem 
%\ref{isometry.cone}). 

 2)  Let $S$ be the isotropy group of $p$ in 
$H$. Since $H$ and $H^{\prime}$ commute, $S$ acts trivially on $H^{\prime}. p$:  $s(h^{\prime}. p)= h^{\prime} s .p = h^{\prime}.p$. In particular, $S$ will act trivially on $T_p(H^{\prime}.p) \cap T_p(H.p)$.
 But $T_p({\mathcal N}_p)$ is the only subspace of $T_p(H_p)$ on which the action is trivial, and since we have already seen that $T_p(H^{\prime}.p)$ must be transverse to $T_p({\mathcal N}_p)$, we get $T_p(H^{\prime}.p) \cap T_p(H.p) = \{0 \}$.
 Now, there is a $S$-invariant splitting 
$T_p M=T_p(H^{\prime}.p) \oplus T_p({\mathcal N}_p) \oplus E$,
 where $E$ is a Riemannian subspace of $T_p(H.p)$, on which $S$ acts irreducibly by the standard action of $O(n-1)$ on $\re^{n-1}$.
 Let us call $F$ the orthogonal of $T_p(H.p)$ in $T_pM$. This space is transverse to $E$, so that $F$ is the graph of a linear map $A:T_p(H^{\prime}.p) \oplus T_p({\mathcal N}_p) \to E$.
 This map $A$ intertwines the trivial action of $S$ on $T_p(H^{\prime}.p) \oplus T_p({\mathcal N}_p)$ with the irreducible one on $E$, so that $A=0$, and $F=T_p(H^{\prime}.p) \oplus T_p({\mathcal N}_p)$. 
As a consequence, the sum $T_p(H^{\prime}.p) \oplus T_p(H.p)$ is orthogonal for the metric $h_p$, and by homogeneity of $M$, this remains true at every point of $M$.    
\end{proof}
\subsection{Proof of corollary \ref{simple.homogeneous}}

Here, we assume that $G$ is semi-simple, noncompact, with finite center.  The group $G$ acts transitively and non-properly on a lightlike manifold $(M,h)$.  Looking at a finite cover of $G$ if necessary, we assume that $G=H_1 \times ... \times H_s$, where each $H_i$ is a simple group with finite center. For $p \in M$, and every $i=1,...,s$, we call $G_p^i$ the projection of the isotropy group $G_p$ on $H_i$, and $H_p^i$ the intersection $G_p \cap H_i$.  Each $H_p^i$ is a normal subgroup of $G_p$. Since $G_p$ is non-compact, some $G_p^i$ has non-compact closure; for example $i=1$. 
Performing a Cartan decomposition of a sequence $(g_k)$ in $G_p$ tending to infinity, and using Proposition \ref{Kowalsky}, we get $X_{1}$ in a Cartan subalgebra of $\lieh_{1}$, and some $p^{\prime} \in M$ such that $W_{X_{1}}$ is isotropic at $p^{\prime}$. If $H_{1}$ is not locally isomorphic to $SL(2,\re)$, we get that $W_{X_{1}}$ has dimension $>1$, and thus $H_{1}.p^{\prime}$ is lightlike and carries a non-proper action of $H_{1}$.  By the previous study, $H_{1}$ is isomorphic to $O(1,n)$, $n \geq 3$, and $H_{1}.p^{\prime}$ is homothetic to $Co^n$.  We can then apply Proposition \ref{produit} to conclude.

We are left with the case where $H_1$ is a finite cover of  $PSL(2,\re)$, and $G_p^1$ does not have compact closure. We claim that  the orbit $H_1.p$ can not have dimension $3$. Indeed, let $< \ >_p$ the pullback of $h_p$ in the Lie algebra $\lieh_1$.  Let $g \in G_p$, and $g_j$ the projection of $g$ on $H_j$. Since $D_pg$ leaves $h_p$ invariant, and the $H_j$'s commute, we get that $< \ >_p$ is $Ad(g_1)$-invariant. But $< \ >_p$ is either Riemannian, or lightlike. In both cases, we saw ( see the proof of Proposition \ref{PSL(2,R)surfaces}) that the subgroup $S \subset H_1$ such that $Ad(S)$ preserves $< \ >_p$  is compact, contradicting the fact that $G_p^1$ does not have compact closure.  

Now, if $H_1.p$ is of dimension $2$ and Riemannian,  $H_p^1$ is a maximal compact subgroup $K \subset H_1$. Now, since $H_p^1$ is normal in $G_p$, we get that $K$ is normal in $G_p^1$, what yields $G_p^1=K$, and a new contradiction. 

We conclude that $H_1.p$ is either $1$-dimensional and lightlike, or $2$-dimensional and lightlike. It follows from Proposition \ref{PSL(2,R)surfaces} that $H_1.p$ is homothetic to a cone $Co^1$ or $Co^2$. We then get the conclusion thanks to Proposition \ref{produit}.   
$\Box$

\section{Proof of theorem \ref{compact}}

Here, we assume that $G$ is simple with finite center, and acts locally faithfully by isometries of a compact lightlike manifold $(M,h)$. 

We first assume, by contradiction, that $G$ is not locally isomorphic to $PSL(2,\re)$.  By compactness , every  sequence of $M$ is non-escaping. It follows from Proposition \ref{Kowalsky} that for every $X$ in a Cartan subalgebra of $\lieg$, $W_X^s$ is isotropic at every $p \in M$. Thus, using the last point of Lemma \ref{reduction.transitive}, and the conclusions of Corollary \ref{simple.homogeneous}, we get that $G$ is locally isomorphic to $O(1,n)$, and any $G$-orbit is homothetic to $Co^n$, $n \geq 3$. Let us call $K$ a maximal compact subgroup of $G$, and let $K_0$ the stabilizer in $K$ of a given point $p_0 \in M$. As we saw it in the proof of Lemma \ref{dimension.orbite.lightlike}, the compact group $K_0$ preserves a Riemannian metric on $M$. Since any Riemannian isometry can be linearized around any fixed point thanks to the exponential map, it is not difficult to prove that the set of fixed points of $K_0$ is a closed  submanifold of  $M$, that we call $M_0$. We 
know explicitely the action of $K$ on $Co^n$, and  observe that every orbit of $K$ is of Riemannian type. Let $S(\liek/\liek_0)$ denote the set of euclidean scalar products on $\liek/\liek_0$. There is a continuous map $\mu: M_0 \to S(\liek/\liek_0)$ defined in the following way: if $ X$ and $Y$ are two vectors of $\liek$, and $\overline X$ and $\overline Y$ are their projections on $\liek/\liek_0$, then $\mu(p)(\overline X,\overline Y)=h_p(X(p),Y(p))$.  Now, on $G.p_0$, there is a  1-parameter flow of homotheties $h^t$, which transforms $h_{|G.p_0}$ into $e^{2t}h_{|G.p_0}$, and commutes with the action of $K$ (in particular, it leaves $M_0 \cap G.p_0$ invariant). From this, it follows that $\mu(h^t.p_0)=e^{2t}\mu(p_0)$. Now, by compactness of $M_0$, there is a sequence $(t_k)$ tending to $+ \infty$ such that $h^{t_k}.p_0$ tends to $p_{\infty} \in M_0$. We should get by continuity of $\mu$: $\lim_{k \to + \infty} e^{2t_k}\mu(p_0)=\mu(p_{\infty})$, which yields the desired 
 contradiction.

It remains to understand what happens if $G$ has finite center, and is locally isomorphic to $PSL(2,\re)$.  Let us fix $X,Y,Z$ a standard basis of $\lieg$: $[Y,Z]=X$,$[X,Y]=2Y$, and $[X,Z]=-2Z$. It follows from Proposition \ref{Kowalsky} that $Y$ and $Z$ are isotropic at every $p \in M$. As a consequence, at any $p \in M$, a non-trivial linear combination of $Y$ and $Z$ has to vanish, so that all the orbits of $G$ are lightlike and have dimension at most 2.  If there is a 2-dimensional orbit $G.p_0$, Proposition \ref{PSL(2,R)surfaces} ensures that it is homothetic to $\re^2 \setminus \{  0\}$ endowed with the metric $d \theta^2$ (namely $Co^2$). We get a contradiction exactly as above, using the action of a maximal compact group and the homothetic flow on $Co^2$  (here $\liek_0=0$).

We conclude that every $G$ orbit is 1-dimensional and lightlike. Since $G$ has finite center, these orbits are finite coverings of the circle, hence closed.  $\Box$

\section{Proof of Theorem \ref{Sub-Lorentz.Theorem}}  
Let us first summarize results on Lorentz dynamics in the following statement, fully proved in \cite{DMZ}, but   early partially proved for instance in \cite{Adams1, Adams2, ADZ, kowalsky}. 

\begin{theorem}   \label{Lorentz.case}

Let $G$ be a semi-simple group with finite center, no compact factor and 
no local factor isomorphic to $SL(2, \re)$,  acting isometrically non-properly on a Lorentz manifold $M$. 
Then, up to a finite cover, $G$ has a factor $G^\prime$ isomorphic to $O(1, n)$ or $O(2, n)$
% and,  all the other factors act properly. 
%Furthermore, there exists a $G^\prime$-
and having some  orbit  homothetic to  $dS_n$ or  $ AdS_n$.
\end{theorem}

Most  developments along the article, in particular Proposition \ref{Kowalsky}, do not explicitly involve  the lightlike
nature  of the ambient metric, and  apply equally  to the Lorentz case, and by the way to the general sub-Lorentz case. This allows  one  to find a non-proper $G$-orbit $O$, i.e. with a stabilizer  algebra containing 
nilpotent elements (see the end of proof of Proposition \ref{Key}). One checks easily that $O$ can not be Riemannian. If $O$ is Lorentz, then, apply Theorem \ref{Lorentz.case} (in the homogeneous case), and if it is lightlike, then apply 
Theorem \ref{simple.transitive}. 
$\Box$

\subsubsection{Some remaining questions} The results of  \cite{DMZ} are 
 stronger than the statement of Theorem \ref{Lorentz.case}, since they contain a detailed geometric description of the Lorentz manifold $M$
 (a warped product structure...). 
 This is the missing part of Theorem \ref{simple.transitive} in the lightlike 
 non-homogeneous case and Theorem \ref{Sub-Lorentz.Theorem} in the 
 sub-Lorentz case. In particular, in this last sub-Lorentz situation, it remains to see whether the manifold is or not pure, i.e. everywhere lightlike, or everywhere Lorentz?

\end{document}